\documentclass[12pt]{amsart}
\usepackage{amsmath}
\usepackage{amssymb}
\usepackage{amsfonts}
\usepackage{amsthm}
\usepackage{xcolor}
\usepackage{enumerate}

\newtheorem{theo}{Theorem}[section]
\newtheorem*{theo*}{Theorem}

\newtheorem{defi}[theo]{Definition}

\newtheorem{exam}[theo]{Example}
\newtheorem{note}[theo]{Note}
\newcommand{\N}{\mathbb{N}}

\title{ On a Ramanujan type expansion of arithmetical functions }

\begin{document}
 
 \author[A Chandran]{Arya Chandran}
 \address{Department of Mathematics, University College, Thiruvananthapuram, Kerala - 695034, India}
 \email{aryavinayachandran@gmail.com}

\author[K V Namboothiri]{K Vishnu Namboothiri}
\address{Department of Mathematics, Government College, Chittur, Palakkad - 678104, INDIA\\Department of Collegiate Education, Government of Kerala, India}
\email{kvnamboothiri@gmail.com}

 \begin{abstract}
 Srinivasa Ramanujan provided series expansions of certain arithmetical functions in terms of the exponential sum defined by $c_r(n)=\sum\limits_{\substack{{m=1}\\(m,r)=1}}^{r}e^{\frac{2 \pi imn}{r}}$ in  [\emph{Trans. Cambridge Philos. Soc, 22(13):259–276, 1918}]. Here we give similar type of expansions in terms of the Cohen-Ramanujan sum defined by E. Cohen in [\emph{Duke Mathematical Journal, 16(85-90):2, 1949}] by $c_r^s(n)=\sum\limits_{\substack{h=1\\(h,r^s)_s=1}}^{r^s}e^{\frac{2\pi i n h}{r^s}}$. We also provide some necessary and sufficient conditions for such expansions to exist. 
 \end{abstract}

 \maketitle
\section{Introduction}
The Ramanujan sum denoted by $c_r(n)$ is defined to be the sum of certain powers of a primitive $r$th root of unity. That is,
\begin{align}
c_r(n)&=\sum\limits_{\substack{{m=1}\\(m,r)=1}}^{r}e^{\frac{2 \pi imn}{r}}
\end{align}
 where $r\in\mathbb{N}$ and $n \in \mathbb{Z}$. This sum appeared for the first time in  a paper of Ramanujan \cite{ramanujan1918certain} where he discussed series expansions of certain arithmetical functions in terms of these sums. These expansions were pointwise convergent. The series expansions he gave there were of the form
 \begin{align}\label{ramsum}
 g(a) = \sum\limits_{\substack{{r=1}}}^{\infty}\widehat{g}(r)c_r(a),
 \end{align} with suitable coefficients $\widehat{g}(r)$. In particular, he gave expansions like 
\begin{align*}
 d(n) = \sum\limits_{\substack{{r=1}}}^{\infty} \frac{\log r}{r}c_r(n)
 \end{align*}
 and
 \begin{align*}
 \sigma(n) = \frac{\pi^2n}{6}\sum\limits_{\substack{{r=1}}}^{\infty} \frac{c_r(n)}{r^2}
\end{align*}
 where $d(n)$ and $\sigma(n)$ are respectively the number of divisors and  the sum of divisors of $n$. Though Ramanujan gave series expansions of some functions, no necessary or sufficient conditions were given by him to understand for what type of functions such an expansion may exist. Attempts on this direction were made by others later.

For an arithmetical function $g$, its mean value is defined by $M(g) = \lim\limits_{\substack{x\rightarrow \infty}}\frac{1}{x} \sum\limits_{\substack{{n \leq x}}} g(n)$, when the limit exists. Carmichael \cite{carmichael1932expansions} proved the following identity  for the Ramanujan sums which helps us to  write down possible candidates for the Ramanujan coefficients of any given arithmetical function.
\begin{theo}[Orthogonality Relation]
$$\lim\limits_{\substack{x\rightarrow \infty}}\frac{1}{x} \sum\limits_{\substack{{n \leq x}}}c_r(n)c_s(n) =  \begin{cases}\phi(r) , \quad\text{ if } r=s\\
	0, \quad\text{ otherwise.}
	\end{cases}$$
	
\end{theo}

By applying the orthogonality relation to ($\ref{ramsum}$), we get 
$\widehat{g}(r) = \frac{M(gc_r)}{\phi(r)}$
 provided the mean value of $gc_r$ exists. Thus Ramanujan expansions exist for those arithmetical functions for which the mean values $M(gc_r)$ exist.
    Wintner proved \cite{hardy1943eratosthenian} later the following sufficient condition for the existence of the mean values.
\begin{theo}
Suppose that $g(n)=\sum\limits_{\substack{{d \mid n}}} f(d) $, and that $\sum\limits_{\substack{{n=1}}}^{\infty} \frac{|f(n)|}{n} < \infty$. Then $M(g)= \sum\limits_{\substack{{n=1}}}^{\infty} \frac{f(n)}{n} $.
\end{theo}
 
Delange \cite{delange1976ramanujan}  improved the above result and proved the following giving another sufficient condition for the Ramanujan expansions to exist. 
\begin{theo}
Suppose that $g(n) = \sum\limits_{\substack{{d \mid n}}} f(d)$, and that $\sum\limits_{\substack{{n=1}}}^{\infty} 2^{w(n)}\frac{|f(n)|}{n} < \infty$, where $w(n)$ is the number of distinct prime divisors of $n$. Then $g$ admits a Ramanujan expansion with $\widehat{g}(q)=\sum\limits_{\substack{{n=1}}}^{\infty}\frac{f(qm)}{qm} $.
\end{theo} 

For a detailed discussion on the above results, please see \cite[Chapter VIII]{schwarz1994arithmetical}. Later, Lucht  \cite{lucht1995ramanujan}  gave an alternate method to compute the Ramanujan coefficients.
\begin{theo}{\cite[Theorem 1]{lucht1995ramanujan} }\label{lucht1}
Let $\widehat{g}: \mathbb{N}\rightarrow \mathbb{C}$ be an arbitrary arithmetical function and $\mu$ the usual M{\"o}bius function. The following are equivalent.
\begin{enumerate}
\item $g(a)= \sum\limits_{\substack{{r=1}}}^{\infty} \widehat{g}(r) c_r(a)$ converges (absolutely) for every $a \in \mathbb{N}$.
\item $\gamma(a)= a \sum\limits_{\substack{{r=1}}}^{\infty} \widehat{g}(ar) \mu(r)$ converges (absolutely) for every $a \in \mathbb{N}$.
\end{enumerate}
\end{theo}
By this theorem, the Ramanujan coefficients can be computed to be $\widehat{g}(n)=\sum\limits_{\substack{a=1\\n|a}}^{\infty}\frac{(\mu*g)(a)}{a}$ where $\mu$ is the usual M\"{o}bius function.
In the same paper, he proved the following theorem to provide Ramanujan expansion to a class of additive functions.
\begin{theo}\label{lucht2}
Let $g \in \mathcal{A}$, the set of all additive arithmetical  functions. If the series $\sum\limits_{\substack{{v=1}}}^{\infty}\frac{g(p^v)}{p^v}$\text{ and } $\sum\limits_{\substack{{p}}}\sum\limits_{\substack{{v=1}}}^{\infty}\frac{g(p^v)}{p^v}$ converge then $g$ has a pointwise convergent Ramanujan expansion (\ref{ramsum}) with coefficients
\begin{align*}
\widehat{g}(p^\alpha) &= \frac{-g(p^{\alpha-1})}{p^{\alpha }}+ (1-\frac{1}{p})\sum\limits_{\substack{{v\geq \alpha}}}\frac{g(p^v)}{p^{v}}\\
\widehat{g}(1) &= \sum\limits_{\substack{{p}}}\widehat{g}(p)\\
\widehat{g}(n) &= 0, \text{ otherwise.}
\end{align*}

\end{theo}

The key component of Ramanujan expansions is the Ramanujan Sum and it has been generalized in many ways. The aim of this paper is to study the Ramanujan type expansions using a generalization of the Ramanujan sum given by E.\ Cohen in \cite{cohen1949extension}. He defined the sum
\begin{align}\label{Coh1}
c_r^s(n)&=\sum\limits_{\substack{{(h,r^s)_s=1}\\h=1}}^{r^s}e^{\frac{2\pi i n h}{r^s}},
\end{align} where $(a,b)_s$ is the generalized gcd of $a$ and $b$  (see the definition in the next section). We call this henceforth as the Cohen-ramanujan sum. When $s=1$, this reduces to the Ramanujan sum.

We will call such expansions by the name Cohen-Ramanujan expansion. We provide expansions for two well known arithmetical functions using this generalization. We also provide some conditions for such expansions to exist following the method of arguments given by Lucht in \cite{lucht1995ramanujan} and \cite{lucht2010survey}. Infact, we will be proving two theorems analogous to theorem \ref{lucht1} and theorem \ref{lucht2} appearing in \cite{lucht1995ramanujan}. Our expansions and results, as in the case of  most of the existing results related to the usual Ramanujan sum, deal only with  pointwise convergence of such expansions.

We would like to remark that some other generalizations also exist for the Ramanujan sum. A few such generalizations were given by Cohen himself \cite{cohen1959trigonometric}, M. Sugunamma \cite{sugunamma1960eckford}, C. S. Venkataraman and Sivaramakrishnan \cite{venkataraman1972extension} and Chidambaraswamy \cite{chidambaraswamy1979generalized}.

 \section{Notations and basic results}
 Most of the notations, functions, and identities we mention in this paper are standard and can be found in \cite{tom1976introduction} or \cite{mccarthy2012introduction}. However for the sake of completeness, we restate some of them below. 
  
  For two arithmetical functions $f$ and $g$, $f*g$ denotes their Dirichlet convolution (Dirichlet product). Then the M{\"o}bius inversion formula states that $f(n)=\sum\limits_{d|n}g(d) \Longleftrightarrow g(n)=\sum\limits_{d|n}f(d)\mu\left(\frac{n}{d}\right)=f*\mu$.

        An arithmetical function $g$ is said to be additive if $g(mn)= g(m)+g(n)$ for coprime positive integers $m$ and $n$.
 $ \mathcal{A}$ denotes the set of all additive arithmetical functions. $\mathcal{P^*}$ denotes the set of all  prime powers $p^\alpha$ with $\alpha \in \mathbb{N}$.

  By $\xi_q^s(n)$, we mean the function $$\xi_q^s(n)= \begin{cases} q^s , \quad\text{ if } q^s\mid n\\
	0, \quad\text{ otherwise.}
	\end{cases}$$ It was proved by Cohen in \cite{cohen1949extension} that 
\begin{align}\label{ram-ident}
\sum\limits_{\substack{r|q}}c_r^{s}(n)=\xi_q^{(s)}(n).
\end{align} 

For $s\in\N$, the generalized GCD function $(a,b)_s$ gives the largest $d^s$ where $d\in\N$ such that $d^s|a$ and $d^s|b$. For $s>1$, a positive integer $m$ is $s-$power free  if no $p^s$ divides $m$, where $p $ is prime.
\begin{defi}
For $s \in \N$, $\tau_s(n)$ gives the number of $d^s\mid n$ where $d^s \in \N$. That is $\tau_s(n)= \sum\limits_{\substack{d^s \mid n\\d \in \N}}1$.
\end{defi}

 For $k,n \in \mathbb{N} $,  $\sigma_k(n) = \sum\limits_{\substack{{d \mid n}}}d^k$, the sum of $k$th powers of the divisors of $n$. 
\begin{defi}
Let $k,s \in \N$. The generalized sum of divisors function $\sigma_{k,s}(n)$ is given by $\sigma_{k,s}(n)=\sum\limits_{\substack{d^s \mid n\\d \in \N}}(d^s)^k$.
\end{defi} 
 Note that this function is different from $\sigma_{ks}$. But $\sigma_{k,1}=\sigma_{k}$.

Using the identity \begin{align}
                    c_r^{s}(n)=\sum\limits_{\substack{d|r\\d^s|n}}\mu(r/d)d^s\label{eq:mu_crs}
                   \end{align}
                   given by Cohen in \cite{cohen1949extension}, we see that for fixed $s, n\in\mathbb{N}$, $c_r^{s}(n)$ is bounded since  $\vert c_r^{s}(n)\vert  \leq \sum\limits_{\substack{d\mid r\\d^s\mid n}}d^s \leq \sigma_{1,s}(n).$  

The unit function $u$ is an arithmetical function such that $u(n)=1$ for all $n$. By $\zeta(s)$, we mean the Riemann Zeta function. By \cite[Example 1,Theorem 11.5]{tom1976introduction}, we have
\begin{align}
 \sum\limits_{\substack{n=1}}^{\infty} \frac{\mu(n)}{n^s}=\frac{1}{\zeta(s)}\text { if }Re(s) > 1.\label{eq:mu_zeta}
 \end{align}

 \section{Main Results}
We begin with giving the Cohen-Ramanujan expansion of $\tau_s$. Here we use elementary number theoretic techniques to establish the result. 

\begin{theo}\label{div-exp}

For $s>1$, we have $\tau_s(n) = \zeta(s)$$ \sum\limits_{\substack{{r=1}}}^{\infty}\frac{c_r^{s}(n)}{r^s}$.
\end{theo}
\begin{proof}
Using (\ref{eq:mu_crs}), we have $c_r^{s}(n) = \sum\limits_{\substack{d\mid r\\d^s\mid n}}\mu(\frac{r}{d})d^s=\sum\limits_{\substack{r=dq\\d^s\mid n}}\mu(q)d^s$.\\
Now consider the sum
\begin{align*}
\sum\limits_{\substack{{r=1}}}^{\infty}\frac{c_r^{s}(n)}{r^s}&=\sum\limits_{\substack{{r=1}}}^{\infty}\sum\limits_{\substack{r=dq\\d^s\mid n}}\frac{\mu(q)d^s}{d^s q^s} = \sum\limits_{\substack{{q=1}}}^{\infty}\sum\limits_{\substack{d^s\mid n}}\frac{\mu(q)}{ q^s}
	 =\sum\limits_{\substack{{q=1}}}^{\infty}\frac{\mu(q)}{ q^s}\sum\limits_{\substack{d^s\mid n}}1
\end{align*}
and it is nothing but $\frac{1}{\zeta(s)} \tau_s(n)$ by identity (\ref{eq:mu_zeta}).
Thus $\tau_s(n) = \zeta(s) \sum\limits_{\substack{{r=1}}}^{\infty}\frac{c_r^{s}(n)}{r^s}$. Since $s>1$ and  $\vert c_r^{s}(n)\vert \leq \sigma_{1,s}(n)$, the sum $ \sum\limits_{\substack{{r=1}}}^{\infty}\frac{c_r^{s}(n)}{r^s}$ converges absolutely.
\end{proof}

\begin{note}
In \cite{ramanujan1918certain}, Ramanujan showed that,  $d(n)=\sum\limits_{\substack{{r=1}}}^{\infty} \frac{\log r}{r}c_r(n)$. Though $\tau_s$ becomes $d$ when $s=1$, the above result cannot be reduced to the Ramanujan's result because  in our case above, we require $s$ to be greater than 1.
\end{note}
 We derive the following Cohen-Ramanujan expansion for $\sigma_{ks}$.

\begin{theo} For $k, s\geq 1$,$\frac{\sigma_{ks}(n)}{n^{ks}} =\zeta((k+1)s) \sum\limits_{\substack{{r=1}}}^{\infty}\frac{c_r^{s}(n^s)}{r^{(k+1)s}}$.
\end{theo}
\begin{proof}

\begin{align*}
\text{We have } \frac{\sigma_{ks}(n)}{n^{ks}}
= \frac{\sum\limits_{\substack{{d\mid n}}}d^{ks}}{n^{ks}}
= \sum\limits_{\substack{{n=dq}}}\frac{(\frac{n}{q})^{ks}}{n^{ks}}&= \sum\limits_{\substack{{q^s\mid n^s}}}\frac{1}{q^{ks}}
= \sum\limits_{\substack{{q=1}}}^{\infty}\frac{1}{q^{ks}}\frac{1}{q^{s}}\xi_q^{(s)}(n^s).
\end{align*}
Now by equation (\ref{ram-ident}), 
\begin{align*}
\frac{\sigma_{ks}(n)}{n^{ks}}=\sum\limits_{\substack{{q=1}}}^{\infty}\frac{1}{q^{(k+1)s}}\sum\limits_{\substack{r \mid q}}c_r^{s}(n^s)
&= \sum\limits_{\substack{{q=1}}}^{\infty}\frac{1}{q^{(k+1)s}}\sum\limits_{\substack{ q=rm}}c_r^{s}(n^s)\\&= \sum\limits_{\substack{{r=1}}}^{\infty}\sum\limits_{\substack{{m=1}}}^{\infty}\frac{1}{r^{(k+1)s}m^{(k+1)s}}c_r^{s}(n^s)\\&= \sum\limits_{\substack{{m=1}}}^{\infty}\frac{1}{m^{(k+1)s}}\sum\limits_{\substack{{r=1}}}^{\infty}\frac{c_r^{s}(n^s)}{r^{(k+1)s}} \\&= \zeta((k+1)s)\sum\limits_{\substack{{r=1}}}^{\infty}\frac{c_r^{s}(n^s)}{r^{(k+1)s}}.
\end{align*}
The above sum converges absolutely since $s>1$ and  $\vert c_r^{s}(n)\vert \leq \sigma_{1,s}(n)$.
\end{proof}
\begin{note}
Since $\sigma_{ks}(n)=\sum\limits_{\substack{{d\mid n}}}d^{ks} = \sum\limits_{\substack{{d^s\mid n^s}}}(d^s)^k=\sigma_{k,s}(n^s) ,$ the above gives an expansion for $\frac{\sigma_{k,s}(n^s)}{n^{ks}}$ also.
\end{note}

\begin{note}
 If $n=m^sn_1$ where $n_1$ is an $s$-power  free positive integer, then  $\sigma_{k,s}(n)=\sigma_{k,s}(m^s)$. Hence $\sigma_{k,s}$ depends only on the $s$-power part in its argument.
\end{note}

\begin{note}
When $s=1$, the above reduces to the expansion \\$\frac{\sigma_k(n)}{n^k}=\zeta(k+1)\sum\limits_{\substack{{r=1}}}^{\infty}\frac{c_r^{}(n)}{r^{k+1}}$ given by Ramanujan in \cite{ramanujan1918certain}.
\end{note}
Our next result is crucial in establishing the existence of the Cohen-Ramanujan expansions for certain class of additive functions.

\begin{theo}\label{Equiv_conditions}
Let $g : \mathbb{N}\rightarrow \mathbb{C}$ be an arbitrary arithmetical function. Then the following are equivalent.\\
$(i) g(a) = \sum\limits_{\substack{{r=1}}}^{\infty} \widehat{g}(r) c_r^{s}(a^s)$ converges absolutely for every $a \in \mathbb{N}$. \\
$(ii) \gamma(a) = a^s \sum\limits_{\substack{{m=1}}}^{\infty} \widehat{g}(am) \mu(m)$ converges absolutely for every $a \in \mathbb{N}$.\\
In case of convergence, $\gamma= \mu * g$.
\end{theo}
\begin{proof}
We have \\$ c_r^{s}(a^s)= \sum\limits_{\substack{d\mid r\\d^s\mid a^s}}\mu(\frac{r}{d})d^s= \sum\limits_{\substack{d^s\mid r^s\\d\mid a}}\mu(\frac{r}{d})d^s= \sum\limits_{\substack{d\mid a}}\mu(\frac{r}{d})\xi_d^{(s)}(r^s)= \sum\limits_{\substack{d\mid a}} f(d),  
$
where $f(d)= \mu(\frac{r}{d})\xi_d^{(s)}(r^s).$

By  M{\"o}bius inversion,

 $\sum\limits_{\substack{d\mid a}} c_r^{s}(d^s) \mu(\frac{a}{d}) = f(a)= \mu(\frac{r}{a})\xi_a^{(s)}(r^s)= \begin{cases} a^s \mu(\frac{r}{a}), \quad\text{ if } a^s\mid r^s\\
	0, \quad\text{ otherwise.}
	\end{cases}\\$

Now we prove that $(i)\Rightarrow (ii)$. 

Suppose $g(a) = \sum\limits_{\substack{{r=1}}}^{\infty} \widehat{g}(r) c_r^{s}(a^s)$ converges absolutely for every $a \in \mathbb{N}$. Then
\begin{align*}
\mu*g(a) = \sum\limits_{\substack{d\mid a}} \mu(\frac{a}{d}) g(d) &= \sum\limits_{\substack{d\mid a}} \mu(\frac{a}{d}) \sum\limits_{\substack{{r=1}}}^{\infty} \widehat{g}(r) c_r^{s}(d^s)\\
&= \sum\limits_{\substack{{r=1}}}^{\infty} \widehat{g}(r) \sum\limits_{\substack{d\mid a}}   c_r^{s}(d^s) \mu(\frac{a}{d})\\
& = \sum\limits_{\substack{{r=1}\\{a^s\mid r^s}}}^{\infty} \widehat{g}(r) a^s \mu(\frac{r}{a})\\
&= \gamma(a) \text{ since }a^s|r^s \Longleftrightarrow a|r.
\end{align*}
From this, we get $\gamma(a) = a^s \sum\limits_{\substack{{m=1}}}^{\infty} \widehat{g}(am)  \mu(m)$ and 
\begin{align*}
a^s \sum\limits_{\substack{{m=1}}}^{\infty} |\widehat{g}(am)  \mu(m)| 
&\leq \sum\limits_{\substack{{r=1}}}^{\infty} \sum\limits_{\substack{d\mid a}}| \widehat{g}(r) \mu(\frac{a}{d}) c_r^{s}(d^s)| \leq \sum\limits_{\substack{d\mid a}} \sum\limits_{\substack{{r=1}}}^{\infty} | \widehat{g}(r) c_r^{s}(d^s)|
\end{align*}
which converges by the assumption. 
Thus $\gamma(a) = a^s \sum\limits_{\substack{{m=1}}}^{\infty} \widehat{g}(am) \mu(m)$ converges absolutely.

To prove that $(ii)\Rightarrow (i)$, suppose that $\gamma(a) = a^s \sum\limits_{\substack{{m=1}}}^{\infty} \widehat{g}(am) \mu(m)$ converges absolutely for every $a \in \mathbb{N}$.
\begin{align*}
\text{Now } u*\gamma(a)  = \sum\limits_{\substack{{d\mid a}}}\gamma(d) u(\frac{a}{d})
= \sum\limits_{\substack{{d\mid a}}}\gamma(d)
&= \sum\limits_{\substack{{d\mid a}}} d^s \sum\limits_{\substack{{m=1}}}^{\infty} \widehat{g}(dm) \mu(m)\\
&= \sum\limits_{\substack{d\mid a\\d\mid r}} d^s \sum\limits_{\substack{{r=1}}}^{\infty} \widehat{g}(r) \mu(\frac{r}{d})\\
&=  \sum\limits_{\substack{{r=1}}}^{\infty} \widehat{g}(r) \sum\limits_{\substack{d\mid a\\d\mid r}} d^s   \mu(\frac{r}{d})\\
&=  \sum\limits_{\substack{{r=1}}}^{\infty} \widehat{g}(r) \sum\limits_{\substack{d^s\mid a^s\\d\mid r}} d^s   \mu(\frac{r}{d})\\
&=  \sum\limits_{\substack{{r=1}}}^{\infty} \widehat{g}(r) c_r^{s}(a^s)
= g(a).
\end{align*}
That is 
\begin{align*}
 g(a)&= \sum\limits_{\substack{{r=1}}}^{\infty} \widehat{g}(r) c_r^{s}(a^s)= \sum\limits_{\substack{{d\mid a}}} d^s \sum\limits_{\substack{{m=1}}}^{\infty} \widehat{g}(dm) \mu(m)
\end{align*}
and it converges absolutely.
\end{proof}

Let us see how to use the above theorem to deal with $\tau_s$.
\begin{exam}
Let  $g(a) = \frac{\tau_s(a^s)}{\zeta(s)} $. By abuse of notation, let $\widehat{g}(r)= \frac{1}{r^s}$. Then
\begin{align*}
\gamma(a) &= a^s \sum\limits_{\substack{{m=1}}}^{\infty} \widehat{g}(am) \mu(m)\\
&=  a^s \sum\limits_{\substack{{m=1}}}^{\infty} \frac{1}{(am)^{s}} \mu(m)\\
&=  \sum\limits_{\substack{{m=1}}}^{\infty} \frac{\mu(m)}{m^{s}}\\
&=  \frac{1}{\zeta(s)} \text{ (by identity }(\ref{eq:mu_zeta}))
\end{align*}
and so $\gamma(a)$ exists. By Theorem \ref{Equiv_conditions}, $ g(a) = \sum\limits_{\substack{{r=1}}}^{\infty} \widehat{g}(r) c_r^{s}(a^s)$ converges absolutely. So 
$\sum\limits_{\substack{{r=1}}}^{\infty}\frac{1}{r^s}c_r^{s}(a^s)=\frac{\tau_s(a^s)}{\zeta(s)}$, giving an expansion for $\tau_s$. This expansion is in agreement with what we proved in Theorem \ref{div-exp}. 
\end{exam}
The same way we can get an expansion for $\sigma_{ks}$.
\begin{exam}
Let $g(a) = \frac{\sigma_{ks}(a)}{a^{ks}} $ and (once again, by abuse of notation) let $\widehat{g}(r)=\frac{\zeta(k+1)}{r^{(k+1)s}}$. Then
\begin{align*}
\gamma(a) &= a^s \sum\limits_{\substack{{m=1}}}^{\infty} \widehat{g}(am) \mu(m)\\
&=  a^s \sum\limits_{\substack{{m=1}}}^{\infty} \frac{\zeta(k+1)}{(am)^{(k+1)s}} \mu(m)\\
&=  \frac{\zeta(k+1)}{(a)^{ks}} \sum\limits_{\substack{{m=1}}}^{\infty}  \frac{\mu(m)}{m^{(k+1)s}}\\
&=  \frac{\zeta(k+1)}{(a)^{ks}}  \frac{1}{\zeta((k+1)s)}  \text{ (by identity }(\ref{eq:mu_zeta})).
\end{align*}
Thus $\gamma(a)$ exists. By Theorem $\ref{Equiv_conditions}$, 
 $g(a) = \sum\limits_{\substack{{r=1}}}^{\infty} \widehat{g}(r) c_r^{s}(a^s)$ converges absolutely. Hence 
 $\sum\limits_{\substack{{r=1}}}^{\infty} \widehat{g}(r) c_r^{s}(a^s)=  \sum\limits_{\substack{{r=1}}}^{\infty}\frac{\zeta(k+1)}{r^{(k+1)s}}c_r^{s}(a^s)=\frac{\sigma_{ks}(a)}{a^{ks}} $ has an absolutely convergent expansion.
\end{exam}
Now we give a sufficient condition for the existence of Cohen-Ramanujan expansions of certain class of additive arithmetical functions.
\begin{theo}
Let $g \in \mathcal{A}$. If the series $\sum\limits_{\substack{{v=1}}}^{\infty}\frac{g(p^v)}{p^{vs}}$ and $\sum\limits_{\substack{{p}}} \sum\limits_{\substack{{v=1}}}^{\infty}\frac{g(p^v)}{p^{vs}}$ converge then $g$ has a pointwise convergent Cohen-Ramanujan expansion with coefficients
\begin{align*}
\widehat{g}(p^\alpha) &= \frac{-g(p^{\alpha-1})}{p^{\alpha s}}+ (1-\frac{1}{p^s})\sum\limits_{\substack{{v\geq \alpha}}}\frac{g(p^v)}{p^{vs}}\\
\widehat{g}(1) &= \sum\limits_{\substack{{p}}}\widehat{g}(p)\\
\widehat{g}(n) &= 0, \text{ otherwise.}
\end{align*}
\end{theo}
\begin{proof}
First we prove the existence of the above Cohen-Ramanujan  coefficients. Since $\sum\limits_{\substack{{v=1}}}^{\infty}\frac{g(p^v)}{p^{vs}}$ converges,
$\widehat{g}(p^\alpha) = \frac{-g(p^{\alpha-1})}{p^{\alpha s}}+ (1-\frac{1}{p^s})\sum\limits_{\substack{{v\geq \alpha}}}\frac{g(p^v)}{p^{vs}}$ exists. Now, since $g$ is additive, $g(1)=0$ and so 
\begin{align*}
\widehat{g}(1) = \sum\limits_{\substack{{p}}}\widehat{g}(p)
&= \sum\limits_{\substack{{p}}}(\frac{-g(1)}{p^s}+ (1-\frac{1}{p^s})\sum\limits_{\substack{{v=1}}}^{\infty}\frac{g(p^v)}{p^{vs}})\text{, where $s\geq 1$ }\\
&= \sum\limits_{\substack{{p}}}(1-\frac{1}{p^s})\sum\limits_{\substack{{v=1}}}^{\infty}\frac{g(p^v)}{p^{vs}} \text{ exists.}
\end{align*}
Also $\widehat{g}(n)=0$ if $n \notin \mathbb{P}^*\cup \{1\}$ exists.\\
Consider the action of $\mu *g$ on prime powers.
\begin{align*}
\mu*g(p^\alpha) = \sum\limits_{\substack{{p^\alpha}}} \mu(d) g(\frac{p^\alpha}{d})
&= \mu(1) g(p^\alpha)+\mu(p)g(p^{\alpha-1})\\
&= g(p^\alpha)-g(p^{\alpha-1}).\\
\text{Since } g \text{ is an additive function, }\mu*g(1)&= \mu(1) g(1)\\
&= 0.
\end{align*}
Let $n = p_1^{r_1}p_2^{r_2}\cdots p_k^{r_k}$,  where $p_i$  are distinct primes. Then,

\begin{align*}
\mu*g(n) &= \sum\limits_{\substack{{d\mid n}}} \mu(d)g(\frac{n}{d})\\
&=\mu(1)g(p_1^{r_1}p_2^{r_2} \cdots p_k^{r_k})+\mu(p_1)g(\frac{p_1^{r_1}p_2^{r_2}\cdots p_k^{r_k}}{p_1})+\\& +\mu(p_2)g(\frac{p_1^{r_1}p_2^{r_2}\cdots p_k^{r_k}}{p_2})+\cdots +\mu(p_k)g(\frac{p_1^{r_1}p_2^{r_2}\cdots p_k^{r_k}}{p_k})\\&+\mu(p_1p_2)g(\frac{p_1^{r_1}p_2^{r_2}\cdots p_k^{r_k}}{p_1p_2})+\mu(p_1p_3)g(\frac{p_1^{r_1}p_2^{r_2}\cdots p_k^{r_k}}{p_1p_3})+\\&+\cdots +\mu(p_{k-1}p_k)g(\frac{p_1^{r_1}p_2^{r_2}\cdots p_k^{r_k}}{p_{k-1}p_k})+\cdots +\mu(p_1p_2\cdots p_k)g(\frac{p_1^{r_1}p_2^{r_2}p_3^{r_3}}{p_1p_2\cdots p_k})\\
&= \sum\limits_{\substack{{i = 1}}}^{k}\left(\binom{k-1}{0}g(p_i^{r^i})-\binom{k-1}{1}g(p_i^{r^i})+\cdots +(-1)^{k-1}\binom{k-1}{k-1}g(p_i^{r^i}) \right)-\\&
 \left(\sum\limits_{\substack{{i = 1}}}^{k}\binom{k-1}{0}g(p_i^{r_i-1})-\binom{k-1}{1}g(p_i^{r_i-1})+\cdots +(-1)^{k-1}\binom{k-1}{k-1}g(p_i^{r_i-1})\right)
\\&= \sum\limits_{\substack{{i = 1}}}^{k}(1+-1)^{k-1}g(p_i)-\sum\limits_{\substack{{i = 1}}}^{k}(1+-1)^{k-1}g(p_i^{r_i-1})
\\&=0.
\end{align*}
Thus $\mu*g(a) = \begin{cases} g(p^\alpha)-g(p^{\alpha-1}), \quad\text{ if } a = p^\alpha \in \mathbb{P}^*\\
	0, \quad\text{ otherwise.}
	\end{cases}$\\
	Now consider the sum $\gamma(a) = a^s \sum\limits_{\substack{{n=1}}}^{\infty} \widehat{g}(an) \mu(n)$.
\begin{align*}
\text{ Then, }\gamma(1) & = \sum\limits_{\substack{{n=1}}}^{\infty}\widehat{g}(n)\mu(n)\\
&= \widehat{g}(1)\mu(1)+\sum\limits_{\substack{{p}}}\widehat{g}(p)\mu(p)\\
&= \widehat{g}(1)-\sum\limits_{\substack{{p}}}\widehat{g}(p)\\
&=0, \text{ by assumption}.
\end{align*}
\begin{align*}
\gamma(p^{\alpha s}) & = p^{\alpha s}\sum\limits_{\substack{{n=1}}}^{\infty}\widehat{g}(p^\alpha n)\mu(n)\\
&= p^{\alpha s}\left(\widehat{g}(p^\alpha)\mu(1)+\widehat{g}(p^\alpha p)\mu(p)\right)\\
&= p^{\alpha s}\left(\widehat{g}(p^\alpha)-\widehat{g}(p^{\alpha+1 })\right)\\
&=p^{\alpha s}\left((\frac{-g(p^{\alpha-1})}{p^{\alpha s}}+ (1-\frac{1}{p^s})\sum\limits_{\substack{{v\geq \alpha}}}\frac{g(p^v)}{p^{vs}})-(\frac{-g(p^{\alpha})}{p^{(\alpha+1) s}}+ (1-\frac{1}{p^s})\sum\limits_{\substack{{v\geq \alpha+1}}}\frac{g(p^v)}{p^{vs}})\right)\\
&= p^{\alpha s}\left((\frac{-g(p^{\alpha-1})}{p^{\alpha s}}+ \frac{g(p^\alpha)}{p^{(\alpha+1)s}}+ (1-\frac{1}{p^s})\frac{g(p^\alpha)}{p^{\alpha s}})\right)\\
&= g(p^{\alpha})-g(p^{\alpha-1}).
\end{align*}
If $a\notin \mathbb{P}^*\cup\{1\}$, then $\widehat{g}(a) = 0$. 
Therefore $\gamma(a) = a^s \sum\limits_{\substack{{n=1}}}^{\infty} \widehat{g}(an) \mu(n) = 0$.\\
Hence $\gamma(a)=\mu*g(a)$, converges absolutely. By Theorem \ref{Equiv_conditions}, $g(a)$ converges absolutely with Cohen-Ramanujan coefficients $\widehat{g}(a)$.
\section{Further directions}
In addition to the expansions mentioned above,  in \cite{ramanujan1918certain}, Ramanujan derived expansion for the Euler totient function $\phi$. It is possible that using our techniques given above, we may get expansions for the Jordan totient function defined as $J_s(n) = n^s \prod\limits_{\substack{{p \mid n}}}(1-\frac{1}{p^s})$ and the Klee's function  defined as $\Phi_s(n) = n\prod\limits_{\substack{p^s\mid n\\p \text{ prime}}}(1-\frac{1}{p^s})$, which behave very much similar to the Euler totient function, but has some $s$th power in their closed form formulae to deal with. We further feel that our techniques  can be used to find such expansions using some other generalizations of the Ramanujan Sum and various other such type of sums.
\section{Acknowledgements}

The first author thanks the University Grants Commission of India for providing financial support for carrying out research work through their Senior Research Fellowship (SRF) scheme. The authors thank the reviewer for offering some insightful comments which made this paper more compact than what it was before.
\section{Data availability statement}
We hereby declare that data sharing not applicable to this article as no datasets were generated or analysed during the current study.

\end{proof}

\end{document}